\documentclass[12pt, reqno]{amsart}
\usepackage{amsmath, amsthm, amscd, amsfonts, amssymb, graphicx, color}
\usepackage[bookmarksnumbered, colorlinks, plainpages]{hyperref}

\newtheorem{theorem}{Theorem}[section]
\newtheorem{lemma}[theorem]{Lemma}

\theoremstyle{definition}

\theoremstyle{remark}

\numberwithin{equation}{section}

\begin{document}
\setcounter{page}{1}

\title[2-Local derivations]{Description of 2-local derivations on some Lie
rings of skew-adjoint matrices}

\author[Ayupov Sh. A.]{\bf Ayupov Sh. A. $^1$}

\address{$^{1}$ Ayupov Shavkat Abdullayevich,
V.I. Romanovskiy Institute of Mathematics, Uzbekistan Academy of Sciences,
Do'rmon yo'li street, Tashkent, 1000125, Uzbekistan.}
\email{\textcolor[rgb]{0.00,0.00,0.84}{ sh$_-$ayupov@mail.ru}}

\author[Arzikulov F. N.]{\bf Arzikulov F. N.$^2$}
\address{$^{2}$ Arzikulov Farhodjon Nematjonovich,
Andizhan State University, Department of Mathematics docent,
University street, Andizhan, 710020, Uzbekistan.}
\email{\textcolor[rgb]{0.00,0.00,0.84}{arzikulovfn@rambler.ru}}

\dedicatory{}

\subjclass[2010]{17B40, 17B65, 16W25, 46L57}

\keywords{derivation, inner Lie derivation, 2-local Lie derivation, Lie ring, Lie
ring of matrices over a commutative ring, skew-adjoint matrix}%

\begin{abstract}
In the present paper we prove that every 2-local inner
derivation on the Lie ring of skew-symmetric matrices over a commutative ring
is an inner derivation. We also apply our technique
to various Lie algebras of infinite dimensional
skew-adjoint matrix-valued maps on a set and prove
that every 2-local spatial derivation on such algebras is a spatial derivation.
\end{abstract}

\maketitle

\section*{Introduction}

In this paper we study 2-local derivations of Lie rings. Investigation of 2-local derivations began with the work \cite{S} of \v{S}emrl, in which he introduced the notion of 2-local derivations and described 2-local derivations on the algebra $B(H)$ of all bounded linear operators on the infinite-dimensional separable Hilbert space $H$. After a number of paper were devoted to 2-local maps on different types of rings, algebras, Banach algebras and Banach spaces. The list of papers devoted to such 2-local maps can be found in the bibliography of \cite{AA3}.

This paper is devoted to the description of 2-local derivations of the Lie ring of skew-adjoint matrices over a commutative ring. By the Poincare-Birkhoff-Witt theorem, every Lie algebra is embedded into some associative algebra. In particular, if in an associative algebra A we take the Lie multiplication $[a,b]=ab-ba$, then we obtain the Lie algebra $(A,[,])$. In this case, every 2-local inner derivation of the algebra $A$ is a derivation if and only if every 2-local inner derivation of the Lie algebra $(A,[,])$ is a derivation (Theorem 1 of \S 2). In general, for any Lie algebra $(L,[,])$, Lie multiplication $[,]$ can be expressed by some associative multiplication. Therefore in many papers the investigation of 2-local Lie derivations is based on associative multiplication. For example, the first paper on 2-local Lie derivations \cite{CLW} of L. Chen, F. Lu and T. Wang is devoted to the description of 2-local Lie derivations of operator algebras over Banach spaces. The first paper devoted to the description of 2-local derivations of Lie algebras that are not isomorphic to associative algebras is the paper \cite{AyuKudRak16} of Sh. Ayupov, K. Kudaybergenov and I. Rakhimov, in which total description of 2-local
Lie derivations for the case of finite dimensional Lie algebras is given. They prove that every 2-local derivation on a finite-dimensional
semi-simple Lie algebra $\mathcal{L}$ over an algebraically closed field of characteristic zero is a derivation.
They also show that a finite-dimensional nilpotent Lie  algebra $\mathcal{L}$ with $\dim \mathcal{L}\geq 2$
admits a 2-local derivation which is not a derivation. At the same time, in \cite{LaiZheng15}
X. Lai and Z.X. Chen give a description of 2-local Lie derivations for the case of finite dimensional
simple Lie algebras. Recently, L. Liu characterized 2-local Lie derivations on a semi-finite
factor von Neumann algebra with dimension greater than 4 in \cite{Liu16}. Finally, in \cite{HLAH} on the algebras
including factor von Neumann algebras, UHF algebras and the Jiang-Su algebra, the authors
prove that every 2-local Lie derivation is a Lie derivation.

In the present paper we consider sufficiently general case of  Lie rings
of skew-symmetric matrices over a commutative ring. Namely, we study inner Lie derivations and 2-local inner
Lie derivations on these matrix Lie rings. We developed an algebraic approach to
investigation of such maps.

In section 2 the relationship between 2-local derivations
and 2-local Lie derivations on associative and Lie rings is studied.

In section 3 we consider 2-local derivations on the Lie ring of matrices over a commutative ring
are studied. Namely, we investigate 2-local inner derivations on the Lie
ring $K_n(\Re)$ of skew-symmetric $n\times n$ matrices over a commutative unital ring $\Re$ and  prove that every such 2-local inner derivation is
an inner derivation. For this propose we develop a Lie analogue of the algebraic approach to
investigation of 2-local derivations applied to matrix
rings over a commutative ring in \cite{AA3}.
As a corollary we established that every 2-local inner Lie derivation on the Lie
algebra $K_n(\mathcal{A})$ of skew-symmetric $n\times n$ matrices over a commutative unital algebra $\mathcal{A}$
is an inner Lie derivation. It should be noted that here we can take a commutative unital algebra over an arbitrary field.
Therefore  the results of \cite{AA3}, \cite{HLAH},  \cite{LaiZheng15} and \cite{Liu16} and
their proofs are  not applicable to obtain the results presented in the present  paper.

In section 4 we also apply this technique
to various Lie algebras of
skew-adjoint infinite dimensional matrix-valued maps on a set and prove
that every 2-local spatial derivation on such algebras is a spatial derivation.

\section{2-local derivations on Lie and associative rings.}

This section is devoted to derivations and 2-local derivations of Lie rings and Lie algebras.

Let $\Re$ be a ring (an algebra) (nonassociative in general). Recall that a (respectively linear) map $D : \Re\to \Re$ is called
a derivation, if $D(x+y)=D(x)+D(y)$ and $D(xy)=D(x)y+xD(y)$ for
any two elements $x$, $y\in \Re$.

A map $\Delta : \Re\to \Re$ is called a 2-local derivation, if for
any two elements $x$, $y\in \Re$ there exists a derivation
$D_{x,y}:\Re\to \Re$ such that $\Delta (x)=D_{x,y}(x)$, $\Delta
(y)=D_{x,y}(y)$.

Let $\Re$ be an associative ring (algebra). A derivation $D$ on $\Re$
is called an inner derivation, if there exists an element $a\in
\Re$ such that
$$
D(x)=ax-xa, x\in \Re.
$$
A map $\Delta : \Re\to \Re$ is called a 2-local inner derivation,
if for any two elements $x$, $y\in \Re$ there exists an element
$a\in \Re$ such that $\Delta (x)=ax-xa$, $\Delta (y)=ay-ya$.

Let $\Re$ be a Lie ring (Lie algebra). Given an an element  $a$ in  $\Re$,
the map $D^L_a(x)=[a,x]-[x,a]$, $x\in \Re$ is a Lie derivation. Such derivation
is called an inner derivation of $\Re$. Let $\Delta$ be a 2-local derivation of $\Re$.
$\Delta$ is called a 2-local inner derivation, if for each pair of elements
$x$, $y\in \Re$ there is an inner derivation $D$ of $\Re$ such that $\Delta(x)=D(x)$, $\Delta(y)=D(y)$.

Let $\mathcal{A}$ be an associative unital ring (algebra over an arbitrary field). Suppose element $2$ is invertible in $\mathcal{A}$.
Then the set $\mathcal{A}$ with respect to the operations of
addition and Lie multiplication
$$
[a,b]=ab-ba, a, b\in \mathcal{A}
$$
is a Lie ring (respectively Lie algebra). This Lie ring (respectively Lie algebra)  will be  denoted by $(\mathcal{A}, [ , ])$. We take $a\in \mathcal{A}$
and the map
$$
D_a^L(x)=[a,x]-[x,a], x\in \mathcal{A}.
$$
This map is additive and moreover it is a derivation. Indeed, we have
$$
D^L_a(x)=[a,x]-[x,a]=(ax-xa)-(xa-ax)=2ax-2xa=D_{2a}(x)
$$
and
$$
D^L_a([x,y])=D_{2a}([x,y])=D_{2a}(x)y+xD_{2a}(y)-(D_{2a}(y)x+yD_{2a}(x))=
$$
$$
[D_{2a}(x),y]+[D_{2a}(y),x]=[D^L_a(x),y]+[x,D^L_a(y)].
$$
Therefore {\it every inner Lie derivation of $(\mathcal{A},[,])$ is an inner derivation of
$\mathcal{A}$.} And also it is easy to see that every inner derivation $D_a=ax-xa$,
$x\in \mathcal{A}$ is an inner Lie derivation of $(\mathcal{A},[,])$. Indeed,
$$
D_a=ax-xa=\frac{1}{2}(ax-xa)-\frac{1}{2}(xa-ax)=[\frac{1}{2}a,x]-[x,\frac{1}{2}a]=D^L_{\frac{1}{2}a}(x).
$$
Let $\Delta$  be a 2-local inner Lie derivation of $(\mathcal{A},[,])$. Then for every pair of elements $x$, $y\in \mathcal{A}$
there is an inner derivation $D$ of $(\mathcal{A},[,])$ such that $\Delta(x)=D(x)$, $\Delta(y)=D(y)$.
But $D$ is also an inner derivation of $\mathcal{A}$. Hence, $\Delta$ is a 2-local inner derivation
of $\mathcal{A}$. So, {\it every 2-local inner derivation of the Lie ring (Lie algebra) $(\mathcal{A},[,])$
is a 2-local inner derivation of the associative ring (respectively associative algebra) $\mathcal{A}$. Conversely, every 2-local inner derivation of the
associative ring (associative algebra) $\mathcal{A}$ is a 2-local inner derivation of the Lie ring (respectively Lie algebra) $(\mathcal{A},[,])$.}
Therefore we have the following theorem.

\begin{theorem} \label{1.1}
Every 2-local inner derivation of the Lie ring (Lie algebra) $(\mathcal{A},[,])$ is a derivation
if and only if every 2-local inner derivation of the
associative ring (respectively associative algebra) $\mathcal{A}$ is a derivation.
\end{theorem}

Let, now, $\mathcal{A}$ be a $*$-ring ($*$-algebra) and $\mathcal{A}_k$ be the set of all skew-adjoint elements of
$\mathcal{A}$. Then, it is known that $(\mathcal{A}_k,[,])$ is a Lie ring (respectively Lie algebra). We take
$a\in\mathcal{A}_k$ and the inner derivation
$$
D^L_a(x)=[a,x]-[x,a], x\in \mathcal{A}_k.
$$
Then
$$
D^L_a(x)=D_{2a}(x), x\in \mathcal{A}_k.
$$
At the same time the map
$$
D_{2a}(x), x\in \mathcal{A}
$$
is an inner derivation on $\mathcal{A}$ and an extension of the derivation $D^L_a$.
Therefore {\it every inner derivation of the Lie ring (Lie algebra) $(\mathcal{A}_k,[,])$ is extended to
an inner derivation of the $*$-ring (respectively $*$-algebra) $\mathcal{A}$.}

As to a 2-local inner derivation, in this
case it is possible to discuss the extension of a 2-local inner derivation of the Lie
ring $(\mathcal{A}_k,[,])$ to a 2-local inner derivation of the involutive ring $\mathcal{A}$ under some
conditions. However, till now it was not possible
to carry out such extension without additional conditions that raises value of the proof of the main result
of the given paper stated below.

\section{2-Local derivations on the Lie ring of skew-symmetric matrices over a commutative ring}

Let $\Re$ be a unital associative ring, $M_n(\Re)$ be the
matrix ring over $\Re$, $n>1$, of matrices of the form
$$
\left[%
\begin{array}{cccc}
 a^{1,1} & a^{1,2} & \cdots & a^{1,n}\\
a^{2,1} & a^{2,2} & \cdots & a^{2,n}\\
\vdots & \vdots & \ddots & \vdots\\
a^{n,1} & a^{n,2} & \cdots & a^{n,n}\\
\end{array}%
\right],
a^{i,j}\in \Re, i,j=1,2,\dots ,n.
$$
Let $\{e_{i,j}\}_{i,j=1}^n$ be the set of matrix units in
$M_n(\Re)$, i.e. $e_{i,j}$ is a matrix with components
$a^{i,j}={\bf 1}$ and $a^{k,l}={\bf 0}$ if $(i,j)\neq(k,l)$, where
${\bf 1}$ is the identity element, ${\bf 0}$ is the zero element of
$\Re$, and a matrix $a\in M_n(\Re)$ is written as $a=\sum_{k,l=1}^n
a^{k,l}e_{k,l}$, where $a^{k,l}\in \Re$ for $k,l=1,2,\dots, n$.

\medskip

Let $\Re$ be a commutative unital ring, $M_n(\Re)$ be an associative ring of $n\times n$ matrices
over $\Re$, $n>1$. Suppose element $2$ is invertible in $\Re$. In this case the set
$$
K_n(\Re)=\{
\left[%
\begin{array}{cccc}
a^{1,1} & a^{1,2} & \cdots & a^{1,n}\\
a^{2,1} & a^{2,2} & \cdots & a^{2,n}\\
\vdots & \vdots & \ddots & \vdots\\
a^{n,1} & a^{n,2} & \cdots & a^{n,n}\\
\end{array}%
\right]
\in M_n(\Re):
$$
$$
a^{i,i}=0,
a^{i,j}=-a^{j,i}, i<j,
i,j=1,2,\dots ,n\}
$$
is a Lie ring with respect to the addition and the Lie multiplication
$$
[a,b]=ab-ba, a, b\in K_n(\Re).
$$
This Lie ring we denote by $K_n(\Re)$.
Throughout this section, let $s_{i,j}=e_{i,j}-e_{j,i}$ for every pair of different indices $i$, $j$
in $\{1,2,\dots ,n\}$.

\begin{lemma}  \label{2.4}
Let $\Re$ be a commutative unital ring, $\Delta$ be a 2-local inner derivation
on $K_n(\Re)$ and let $D^L_a$, $D^L_b$ be inner derivations on $K_n(\Re)$,
generated by $a$, $b\in K_n(\Re)$ such that
$$
\Delta (s_{i,j})=D^L_a(s_{i,j})=D^L_b(s_{i,j})
$$
for an arbitrary pair of different indices $i$, $j$ such that $i<j$.
Then
$$
a^{ki}=b^{ki}, a^{kj}=b^{kj}, a^{ik}=b^{ik}, a^{jk}=b^{jk}.
$$
for all $k$ distinct from $i$ and $j$.
\end{lemma}

\begin{proof}
By the definition of an inner derivation
$$
D^L_a(s_{i,j})=D_{2a}(s_{i,j}),
$$
$$
D^L_b(s_{i,j})=D_{2b}(s_{i,j}).
$$
Hence, since $D^L_a(s_{i,j})=D^L_b(s_{i,j})$ we have
$$
D_{2a}(s_{i,j})=D_{2b}(s_{i,j}),
$$
i.e.
$$
as_{i,j}-s_{i,j}a=bs_{i,j}-s_{i,j}b.
$$
Then
$$
as_{i,j}-bs_{i,j}=s_{i,j}a-s_{i,j}b\,\,and\,\,(a-b)s_{i,j}=s_{i,j}(a-b).
$$
Hence
$$
(a-b)s_{i,j}^2=s_{i,j}(a-b)s_{i,j}.
$$
From this follows that
$$
a^{i,k}=b^{i,k}, a^{j,k}=b^{j,k}
$$
when $j<k$,
$$
a^{i,k}=b^{i,k}, a^{k,j}=b^{k,j}
$$
when $i<k<j$ and
$$
a^{k,i}=b^{k,i}, a^{k,j}=b^{k,j}
$$
when $k<i$. The proof is complete.
\end{proof}

\begin{lemma}  \label{2.5}
Let $\Re$ be a commutative unital ring, $\Delta$ be a 2-local inner derivation
on $K_n(\Re)$. Then there exist $a\in K_n(\Re)$ such that
$$
\Delta (s_{i,j})=D^L_a(s_{i,j})
$$
for every pair of different indices $i$, $j$.
\end{lemma}

\begin{proof}
By the definition of 2-local inner derivation
there exists a set $\{a(i,j)\}_{i<j}\subset K_n(\Re)$ such that
$$
\Delta (s_{1,2})=D^L_{a(1,2)}(s_{1,2})=D^L_{a(1,3)}(s_{1,2}),\Delta (s_{1,3})=D^L_{a(1,3)}(s_{1,3})=D^L_{a(1,4)}(s_{1,3}),
$$
$$
\Delta (s_{1,4})=D^L_{a(1,4)}(s_{1,4})=D^L_{a(1,5)}(s_{1,4}),...,\Delta (s_{1,n})=D^L_{a(1,n)}(s_{1,n})=D^L_{a(2,3)}(s_{1,n}),
$$
$$
\Delta (s_{2,3})=D^L_{a(2,3)}(s_{2,3})=D^L_{a(2,4)}(s_{2,3}),\Delta (s_{2,4})=D^L_{a(2,4)}(s_{2,4})=D^L_{a(2,5)}(s_{2,4}),
$$
$$
\Delta (s_{2,5})=D^L_{a(2,5)}(s_{2,5})=D^L_{a(2,6)}(s_{2,5}),...,\Delta (s_{2,n})=D^L_{a(2,n)}(s_{2,n})=D^L_{a(3,4)}(s_{2,n}),
$$
$$
..........................................................................
$$
$$
\Delta (s_{n-2,n-1})=D^L_{a(n-2,n-1)}(s_{n-2,n-1})=D^L_{a(n-2,n)}(s_{n-2,n-1}),
$$
$$
\Delta (s_{n-2,n})=D^L_{a(n-2,n)}(s_{n-2,n})=D^L_{a(n-1,n)}(s_{n-2,n}),
$$
$$
\Delta (s_{n-1,n})=D^L_{a(n-1,n)}(s_{n-1,n})=D^L_{a(1,2)}(s_{n-1,n}).
$$
Then by lemma \ref{2.4} for each pair $i$, $j$ of indices such that $i<j<n$ we have
$$
a(i,j)^{i,k}=a(i,j+1)^{i,k}, a(i,j)^{j,k}=a(i,j+1)^{j,k}
$$
for all $k$ distinct from $i$ and $j$,
$$
a(i,n)^{i,k}=a(i,i+1)^{i,k}, a(i,n)^{n,k}=a(i,i+1)^{n,k}
$$
for all $k$ distinct from $i$ and $n$ and
$$
a(n-1,n)^{n-1,k}=a(1,2)^{n-1,k}, a(n-1,n)^{n,k}=a(1,2)^{n,k}
$$
for all $k$ distinct from $n-1$ and $n$.
Hence
$$
a(1,2)=a(1,3)=...=a(1,n)=
$$
$$
a(2,3)=a(2,4)=...=a(2,n)=...=a(n-1,n).
$$
The proof is complete.
\end{proof}

\begin{theorem} \label{2.6}
Let $\Re$ be a commutative unital ring, and let $K_n(\Re)$, $n>3$,
be the Lie ring of skew-symmetric $n\times n$ matrices over $\Re$. Then
any 2-local inner derivation on the matrix Lie ring $K_n(\Re)$ is an inner derivation.
\end{theorem}

\begin{proof}
Let $\Delta: K_n(\Re)\to K_n(\Re)$ be an arbitrary 2-local derivation.
By lemma \ref{2.5} we have an element $a\in K_n(\Re)$ such that
$$
\Delta(s_{i,j})=D^L_a(s_{i,j})
$$
for every pair of different indices $i$, $j$ satisfying $i<j$.

Let $x$ be an arbitrary element in $K_n(\Re)$ and $b$ be an element in $K_n(\Re)$ such that
$$
\Delta(s_{1,2})=D^L_b(s_{1,2}),
\Delta(x)=D^L_b(x)
$$
Then
$$
a^{1,k}=b^{1,k}, a^{2,k}=b^{2,k}.
$$
for all $k$ distinct from $1$ and $2$ by lemma \ref{2.4}.
Hence
$$
e_{1,1}\Delta(x)e_{2,2}+e_{2,2}\Delta(x)e_{1,1}=e_{1,1}2(bx-xb)e_{2,2}+e_{2,2}2(bx-xb)e_{1,1}=
$$
$$
\sum_{k=3}^n2[x^{1,k}b^{2,k}-b^{1,k}x^{2,k}]s_{1,2}=\sum_{k=3}^n2[x^{1,k}a^{2,k}-a^{1,k}x^{2,k}]s_{1,2}.
$$

Let $i$, $j$ be an arbitrary pair of different indices satisfying $i<j$ and $c$ be an element in $K_n(\Re)$ such that
$$
\Delta(s_{i,j})=D^L_c(s_{i,j}),
\Delta(x)=D^L_c(x)
$$
Then
$$
a^{i,k}=c^{i,k}, a^{j,k}=c^{j,k}.
$$
for all $k$ distinct from $i$ and $j$ by lemma \ref{2.4}.
Hence we have
$$
e_{i,i}\Delta(x)e_{j,j}+e_{j,j}\Delta(x)e_{i,i}=e_{i,i}2(bx-xb)e_{j,j}+e_{j,j}2(bx-xb)e_{i,i}=
$$
$$
\sum_{k\neq i,k\neq j}^n2[x^{i,k}b^{j,k}-b^{i,k}x^{j,k}]s_{i,j}=\sum_{k\neq i,k\neq j}^n2[x^{i,k}a^{j,k}-a^{i,k}x^{j,k}]s_{i,j}.
$$
Hence
$$
\Delta(x)=\sum_{i<j}[e_{i,i}\Delta(x)e_{j,j}+e_{j,j}\Delta(x)e_{i,i}]=
$$
$$
\sum_{i<j}[\sum_{k\neq i,k\neq j}2[x^{i,k}a^{j,k}-a^{i,k}x^{j,k}]s_{i,j}]=
$$
$$
2(ax-xa)=[a,x]-[x,a]=D^L_a(x).
$$
Since $x$ was arbitrarily chosen, $\Delta$ is an inner derivation. This completes the proof.
\end{proof}

The proof of Theorem \ref{2.6} is also valid for Lie algebras of skew-symmetric matrices over
a commutative algebra. Therefore we have the following statement.

\begin{theorem} \label{2.7}
Let $\mathcal{A}$ be a commutative unital algebra, and let $K_n(\mathcal{A})$, $n>3$,
be the Lie algebra of skew-symmetric $n\times n$ matrices over $\mathcal{A}$. Then
any 2-local inner derivation on the matrix Lie algebra $K_n(\mathcal{A})$ is an inner derivation.
\end{theorem}

\section{2-local derivations on Lie algebras of skew-adjoint matrix-valued
maps}

Recall that a Hilbert space $H$ is an infinite dimensional inner product space which
is a complete metric space with respect to the metric generated by the inner product \cite[Section 1.5]{AG}.

Throughout this section, let $n$ be an arbitrary cardinal
number and, let $\Xi$ be a set of indices of cardinality $n$. Now let
$H$ be a real Hilbert space of dimension $n$ and, let $B(H)$ be the real von Neumann algebra of
all bounded linear operators on $H$. Let $\{e_i\}_{i\in \Xi}$ be a maximal family of orthogonal minimal
projections in $B(H)$ and, let $\{e_{i,j}\}_{i,j\in \Xi}$ be the family of matrix units defined
by $\{e_i\}_{i\in \Xi}$, i.e. $e_{i,i}=e_i$, $e_{i,i}=e_{i,j}e_{j,i}$ and $e_{j,j}=e_{j,i}e_{i,j}$ for
each pair $i$, $j$ of indices from $\Xi$.

Let $B_{sk}(H)$ be the vector space of all skew-adjoint matrices in $B(H)$, i.e.
$$
B_{sk}(H)=\{a\in B(H): a^*=-a\}.
$$
Then with respect to Lie multiplication
$$
[a,b]=ab-ba, a,b\in B_{sk}(H)
$$
$B_{sk}(H)$ is a Lie algebra.

Let $\Omega$ be an arbitrary set, $M(\Omega,B_{sk}(H))$ be the
Lie algebra of all maps of $\Omega$ to $B_{sk}(H)$. Put
$$
\hat{e}_{i,j}=\sum_{\xi,\eta\in \Xi}\lambda^{\xi,\eta}e_{\xi,\eta},
$$
where for all $\xi$, $\eta$, if $\xi=i$, $\eta=j$ then
$\lambda^{\xi,\eta}={\bf 1}$, else $\lambda^{\xi,\eta}=0$, ${\bf 1}$
is unit of the algebra $F(\Omega)$ of all real number-valued maps on $\Omega$.
Let $s_{i,j}=\hat{e}_{i,j}-\hat{e}_{j,i}$ for all
distinct $i$, $j\in \Xi$.

\medskip

{\it Definition.} Let $A$ be a Lie algebra and, let $B$ be a Lie subalgebra of $A$.
A derivation $D$ on $B$ is said to be spatial, if $D$ is implemented by an element in $A$,
i.e.
$$
D(x)=ax-xa, x\in B,
$$
for some $a\in A$. A 2-local derivation $\Delta$ on $B$ is called 2-local spatial derivation
with respect to derivations implemented by an element in $A$, if
for every two elements $x$, $y\in B$ there exists an element
$a\in A$ such that $\Delta(x)=ax-xa$, $\Delta(y)=ay-ya$.

\begin{lemma}  \label{3.2}
Let $\Omega$ be an arbitrary set, $M(\Omega,B_{sk}(H))$ be the
Lie algebra of all maps of $\Omega$ to $B_{sk}(H)$. Let
$\mathcal{L}$ be a Lie subalgebra of $M(\Omega,B_{sk}(H))$ containing the family
$\{s_{i,j}\}_{i,j\in \Xi}$ and $\Delta$ be a 2-local spatial derivation
on $\mathcal{L}$
with respect to derivations implemented by an element in $M(\Omega,B_{sk}(H))$
and $D^L_a$, $D^L_b$ be spatial derivations on $\mathcal{L}$,
implemented by $a$, $b\in M(\Omega,B_{sk}(H))$, respectively, such that
$$
\Delta (s_{i,j})=D^L_a(s_{i,j})=D^L_b(s_{i,j})
$$
for an arbitrary pair of different indices $i$, $j$. Then
$$
(1-(\hat{e}_{i,i}+\hat{e}_{j,j}))a\hat{e}_{i,i}=(1-(\hat{e}_{i,i}+\hat{e}_{j,j}))b\hat{e}_{i,i},
$$
$$
(1-(\hat{e}_{i,i}+\hat{e}_{j,j}))a\hat{e}_{j,j}=(1-(\hat{e}_{i,i}+\hat{e}_{j,j}))b\hat{e}_{j,j},
$$
$$
\hat{e}_{i,i}a(1-(\hat{e}_{i,i}+\hat{e}_{j,j}))=\hat{e}_{i,i}b(1-(\hat{e}_{i,i}+\hat{e}_{j,j})),
$$
$$
\hat{e}_{j,j}a(1-(\hat{e}_{i,i}+\hat{e}_{j,j}))=\hat{e}_{j,j}b(1-(\hat{e}_{i,i}+\hat{e}_{j,j})).
$$
\end{lemma}

\begin{proof}
By the definition of an inner derivation
$$
D^L_a(s_{i,j})=D_{2a}(s_{i,j}),
$$
$$
D^L_b(s_{i,j})=D_{2b}(s_{i,j}).
$$
Hence, since $D^L_a(s_{i,j})=D^L_b(s_{i,j})$ we have
$$
D_{2a}(s_{i,j})=D_{2b}(s_{i,j}),
$$
i.e.
$$
as_{i,j}-s_{i,j}a=bs_{i,j}-s_{i,j}b.
$$
It follows from this that
$$
s_{i,j}a(1-(\hat{e}_{i,i}+\hat{e}_{j,j}))=s_{i,j}b(1-(\hat{e}_{i,i}+\hat{e}_{j,j}))
$$
since $s_{i,j}(1-(\hat{e}_{i,i}+\hat{e}_{j,j}))=0$.
Hence
$$
\hat{e}_{i,i}a(1-(\hat{e}_{i,i}+\hat{e}_{j,j}))=\hat{e}_{i,i}b(1-(\hat{e}_{i,i}+\hat{e}_{j,j})),
$$
$$
\hat{e}_{j,j}a(1-(\hat{e}_{i,i}+\hat{e}_{j,j}))=\hat{e}_{j,j}b(1-(\hat{e}_{i,i}+\hat{e}_{j,j})).
$$
Similarly,
$$
(1-(\hat{e}_{i,i}+\hat{e}_{j,j}))as_{i,j}=(1-(\hat{e}_{i,i}+\hat{e}_{j,j}))bs_{i,j}
$$
since $(1-(\hat{e}_{i,i}+\hat{e}_{j,j}))s_{i,j}=0$. Hence
$$
(1-(\hat{e}_{i,i}+\hat{e}_{j,j}))a\hat{e}_{i,i}=(1-(\hat{e}_{i,i}+\hat{e}_{j,j}))b\hat{e}_{i,i},
$$
$$
(1-(\hat{e}_{i,i}+\hat{e}_{j,j}))a\hat{e}_{j,j}=(1-(\hat{e}_{i,i}+\hat{e}_{j,j}))b\hat{e}_{j,j}.
$$
\end{proof}

\begin{lemma}  \label{3.3}
Let $\Omega$ be an arbitrary set, $M(\Omega,B_{sk}(H))$ be the
Lie algebra of all maps of $\Omega$ to $K_n({\mathbb R})$. Let
$\mathcal{L}$ be a Lie subalgebra of $M(\Omega,B_{sk}(H))$ containing the family
$\{s_{i,j}\}_{i,j\in \Xi}$ and $\Delta$ be a 2-local spatial derivation
on $\mathcal{L}$ with respect to derivations implemented by an element in $M(\Omega,B_{sk}(H))$.
Then there exist $a\in M(\Omega,B_{sk}(H))$ such that
$$
\Delta (s_{i,j})=D^L_a(s_{i,j})
$$
for every pair of different indices $i$, $j$.
\end{lemma}

\begin{proof}
By the definition of 2-local spatial derivation
on $\mathcal{L}$ with respect to derivations implemented by an element in $M(\Omega,B_{sk}(H))$
there exists a set $\{a(i,j,k,l)\}\subset M(\Omega,B_{sk}(H))$ such that
$$
\Delta (s_{i,j})=D^L_{a(i,j,k,l)}(s_{i,j}),\Delta (s_{k,l})=D^L_{a(i,j,k,l)}(s_{k,l})
$$
for every distinct pairs of indices $(i,j)$ and $(k,l)$.
Then by lemma \ref{3.2}
$$
\hat{e}_{i,i}a(i,j,k,l)(1-(\hat{e}_{i,i}+\hat{e}_{j,j}))=\hat{e}_{i,i}a(i,j,\xi,\eta)(1-(\hat{e}_{i,i}+\hat{e}_{j,j})),
$$
$$
(1-(\hat{e}_{i,i}+\hat{e}_{j,j}))a(i,j,k,l)\hat{e}_{i,i}=(1-(\hat{e}_{i,i}+\hat{e}_{j,j}))a(i,j,\xi,\eta)\hat{e}_{i,i},
$$
$$
\hat{e}_{j,j}a(i,j,k,l)(1-(\hat{e}_{i,i}+\hat{e}_{j,j}))=\hat{e}_{j,j}a(i,j,\xi,\eta)(1-(\hat{e}_{i,i}+\hat{e}_{j,j})),
$$
$$
(1-(\hat{e}_{i,i}+\hat{e}_{j,j}))a(i,j,k,l)\hat{e}_{j,j}=(1-(\hat{e}_{i,i}+\hat{e}_{j,j}))a(i,j,\xi,\eta)\hat{e}_{j,j}
$$
for all pairs $(i,j)$, $(k,l)$ and $(\xi,\eta)$ of mutually distinct indices. At the same time
$$
\hat{e}_{k,k}a(i,j,k,l)(1-(\hat{e}_{k,k}+\hat{e}_{l,l}))=\hat{e}_{k,k}a(k,l,\xi,\eta)(1-(\hat{e}_{k,k}+\hat{e}_{l,l})),
$$
$$
(1-(\hat{e}_{k,k}+\hat{e}_{l,l}))a(i,j,k,l)\hat{e}_{k,k}=(1-(\hat{e}_{k,k}+\hat{e}_{l,l}))a(k,l,\xi,\eta)\hat{e}_{k,k},
$$
$$
\hat{e}_{l,l}a(i,j,k,l)(1-(\hat{e}_{k,k}+\hat{e}_{l,l}))=\hat{e}_{l,l}a(k,l,\xi,\eta)(1-(\hat{e}_{k,k}+\hat{e}_{l,l})),
$$
$$
(1-(\hat{e}_{k,k}+\hat{e}_{l,l}))a(i,j,k,l)\hat{e}_{l,l}=(1-(\hat{e}_{k,k}+\hat{e}_{l,l}))a(k,l,\xi,\eta)\hat{e}_{l,l}.
$$
Hence, if $i\neq k$, $j=l$ then
$$
\hat{e}_{j,j}a(i,j,k,j)\hat{e}_{i,i}=\hat{e}_{j,j}a(k,j,\xi,\eta)\hat{e}_{i,i},
\hat{e}_{i,i}a(i,j,k,j)\hat{e}_{j,j}=\hat{e}_{i,i}a(k,j,\xi,\eta)\hat{e}_{j,j},
$$
and, if $i=k$, $j\neq l$ then
$$
\hat{e}_{i,i}a(i,j,i,l)\hat{e}_{j,j}=\hat{e}_{i,i}a(i,l,\xi,\eta)\hat{e}_{j,j},
\hat{e}_{j,j}a(i,j,i,l)\hat{e}_{i,i}=\hat{e}_{j,j}a(i,l,\xi,\eta)\hat{e}_{i,i}.
$$
So
$$
\hat{e}_{i,i}a(i,j,k,l)(1-\hat{e}_{i,i})=\hat{e}_{i,i}a(i,j,\xi,\eta)(1-\hat{e}_{i,i}),
$$
$$
(1-\hat{e}_{i,i})a(i,j,k,l)\hat{e}_{i,i}=(1-\hat{e}_{i,i})a(i,j,\xi,\eta)\hat{e}_{i,i},
$$
$$
\hat{e}_{j,j}a(i,j,k,l)(1-\hat{e}_{j,j})=\hat{e}_{j,j}a(i,j,\xi,\eta)(1-\hat{e}_{j,j}),
$$
$$
(1-\hat{e}_{j,j})a(i,j,k,l)\hat{e}_{j,j}=(1-\hat{e}_{j,j})a(i,j,\xi,\eta)\hat{e}_{j,j}
$$
for all pairs $(i,j)$, $(k,l)$ and $(\xi,\eta)$ of mutually distinct indices. Therefore
$$
a(i,j,k,l)=a(\alpha,\beta,\xi,\eta)
$$
for all pairs $(i,j)$, $(k,l)$, $(\alpha,\beta)$ and $(\xi,\eta)$ of mutually distinct indices. The proof is complete.
\end{proof}

The following theorem is the key result of this section.

\begin{theorem} \label{3.4}
Let $\Omega$ be an arbitrary set, $M(\Omega,B_{sk}(H))$ be the
Lie algebra of all maps from $\Omega$ to $B_{sk}(H)$. Let
$\mathcal{L}$ be a Lie subalgebra of $M(\Omega,B_{sk}(H))$ containing the family
$\{s_{i,j}\}_{i,j\in \Xi}$. Then any 2-local spatial derivation
on $\mathcal{L}$ with respect to derivations implemented by an element in $M(\Omega,B_{sk}(H))$
is a spatial derivation.
\end{theorem}

\begin{proof}
Let $\Delta: \mathcal{L}\to \mathcal{L}$ be an arbitrary 2-local spatial derivation
on $\mathcal{L}$ with respect to derivations implemented by an element in $M(\Omega,B_{sk}(H))$.
By lemma \ref{3.3} there exists an element $a\in M(\Omega,B_{sk}(H))$ such that
$$
\Delta(s_{i,j})=D^L_a(s_{i,j})
$$
for every pair of different indices $i$, $j$.

Let $i$, $j$ be an arbitrary pair of different indices, $x$ be an arbitrary element in $\mathcal{L}$ and $b$ be an element in $\mathcal{L}$ such that
$$
\Delta(s_{i,j})=D^L_b(s_{i,j}),
\Delta(x)=D^L_b(x)
$$
Then
$$
\hat{e}_{i,i}a(1-(e_{i,i}+e_{j,j}))=\hat{e}_{i,i}b(1-(e_{i,i}+e_{j,j})), \hat{e}_{j,j}a(1-(e_{i,i}+e_{j,j}))=\hat{e}_{j,j}b(1-(e_{i,i}+e_{j,j})),
$$
$$
(1-(e_{i,i}+e_{j,j}))a\hat{e}_{i,i}=(1-(e_{i,i}+e_{j,j}))b\hat{e}_{i,i}, (1-(e_{i,i}+e_{j,j}))a\hat{e}_{j,j}=(1-(e_{i,i}+e_{j,j}))b\hat{e}_{j,j}
$$
by lemma \ref{3.2}.
For each pair $\alpha$, $\beta$ of indices there exist $\lambda_{\alpha,\beta}$, $\mu_{\alpha,\beta}$ in the commutative algebra
of all maps from $\Omega$ to ${\Bbb R}$ such that
$$
e_{\alpha,\alpha}ae_{\beta,\beta}=\lambda_{\alpha,\beta}e_{\alpha,\beta}\,\,and\,\,e_{\beta,\beta}ae_{\alpha,\alpha}=-\lambda_{\alpha,\beta}e_{\beta,\alpha}
$$
and
$$
e_{\alpha,\alpha}be_{\beta,\beta}=\mu_{\alpha,\beta}e_{\alpha,\beta}\,\,and\,\,e_{\beta,\beta}be_{\alpha,\alpha}=-\mu_{\alpha,\beta}e_{\beta,\alpha}.
$$
Therefore
$$
e_{i,i}a(e_{i,i}+e_{j,j})xe_{j,j}-e_{i,i}x(e_{i,i}+e_{j,j})ae_{j,j}+
$$
$$
e_{j,j}a(e_{i,i}+e_{j,j})xe_{i,i}-e_{j,j}x(e_{i,i}+e_{j,j})ae_{i,i}=0
$$
and
$$
e_{i,i}b(e_{i,i}+e_{j,j})xe_{j,j}-e_{i,i}x(e_{i,i}+e_{j,j})be_{j,j}+
$$
$$
e_{j,j}b(e_{i,i}+e_{j,j})xe_{i,i}-e_{j,j}x(e_{i,i}+e_{j,j})be_{i,i}=0.
$$
So, by Lemma \ref{3.2}
$$
e_{i,i}\Delta(x)e_{j,j}+e_{j,j}\Delta(x)e_{i,i}=e_{i,i}2(bx-xb)e_{j,j}+e_{j,j}2(bx-xb)e_{i,i}=
$$
$$
2e_{i,i}b(e_{i,i}+e_{j,j})xe_{j,j}-2e_{i,i}x(e_{i,i}+e_{j,j})be_{j,j}+
$$
$$
2e_{j,j}b(e_{i,i}+e_{j,j})xe_{i,i}-2e_{j,j}x(e_{i,i}+e_{j,j})be_{i,i}+
$$
$$
2e_{i,i}b(1-(e_{i,i}+e_{j,j}))xe_{j,j}-2e_{i,i}x(1-(e_{i,i}+e_{j,j}))be_{j,j}+
$$
$$
2e_{j,j}b(1-(e_{i,i}+e_{j,j}))xe_{i,i}-2e_{j,j}x(1-(e_{i,i}+e_{j,j}))be_{i,i}=
$$
$$
2e_{i,i}a(e_{i,i}+e_{j,j})xe_{j,j}-2e_{i,i}x(e_{i,i}+e_{j,j})ae_{j,j}+
$$
$$
2e_{j,j}a(e_{i,i}+e_{j,j})xe_{i,i}-2e_{j,j}x(e_{i,i}+e_{j,j})ae_{i,i}+
$$
$$
2e_{i,i}a(1-(e_{i,i}+e_{j,j}))xe_{j,j}-2e_{i,i}x(1-(e_{i,i}+e_{j,j}))ae_{j,j}+
$$
$$
2e_{j,j}a(1-(e_{i,i}+e_{j,j}))xe_{i,i}-2e_{j,j}x(1-(e_{i,i}+e_{j,j}))ae_{i,i}=
$$
$$
2(ax-xa)=[a,x]-[x,a]=D^L_a(x).
$$
Since $x$ is chosen arbitrarily $\Delta$ is a spatial derivation. This completes the proof.
\end{proof}

In particular, theorem \ref{3.4} contains the following theorem.

\begin{theorem} \label{3.5}
Let $\Omega$ be an arbitrary set, $M(\Omega,B_{sk}(H))$ be the
Lie algebra of all maps of $\Omega$ to $B_{sk}(H)$.
Then any 2-local inner derivation on the Lie algebra $M(\Omega,B_{sk}(H))$ is an inner derivation.
\end{theorem}

Let $X$ be a hyperstonean compact such that $supp(X)=\Omega$, $C(X)$
denotes the algebra of all ${\mathbb R}$-valued continuous maps on $X$.
There exists a subalgebra $\mathcal{N}$ in $M(\Omega,B_{sk}(H))$ which is a real von Neumann algebra
with the center isomorphic to $C(X)$ (see \cite[Page 12]{AFN3}). More precisely
$\mathcal{N}$ is a real von Neumann algebra of type I.

Similar to $K_n({\mathbb R})$, the vector space
$$
\mathcal{K}=\{a\in \mathcal{M}: a^*=-a\}
$$
of all skew-adjoint elements in $\mathcal{M}$
is a Lie algebra with respect to Lie multiplication
$$
[a,b]=ab-ba, a,b\in \mathcal{K}.
$$
The Lie algebra $\mathcal{K}$ is a Lie subalgebra of $M(\Omega,B_{sk}(H))$
containing the family $\{s_{i,j}\}_{i,j\in \Xi}$. Hence, by theorem \ref{3.4} we have the following result.

\begin{theorem} \label{3.5}
Every 2-local spatial derivation
on $\mathcal{K}$ with respect to derivations implemented by an element in $M(\Omega,B_{sk}(H))$
is a spatial derivation.
\end{theorem}

Let $\mathcal{K}(H)$ be the C$^*$-algebra of all compact
operators on the Hilbert space $H$ over ${\mathbb R}$, $\mathcal{K}_{K}(H)$ be
the Lie algebra of all skew-adjoint compact operators on the Hilbert space $H$.
Let $Q$ be a topological space such that $supp(Q)=\Omega$. Then
the vector space $C(Q,\mathcal{K}_{K}(H))$ of all continuous maps
of $Q$ to $\mathcal{K}_{K}(H)$ is a Lie subalgebra of $M(\Omega,B_{sk}(H))$
containing the family $\{s_{i,j}\}_{i,j\in \Xi}$. Therefore by theorem \ref{3.4} we have
the following result.

\begin{theorem} \label{3.6}
Every 2-local spatial derivation
on the Lie algebra $C(Q,\mathcal{K}_{K}(H))$ with respect to derivations implemented by an element in $M(\Omega,B_{sk}(H))$
is a spatial derivation.
\end{theorem}

\end{document}